\documentclass[twoside]{article}
\usepackage{graphicx,amssymb,mathrsfs,amsmath,enumerate}
\usepackage{amsthm}
\usepackage{charter}
\usepackage[dvipdfm]{hyperref}
\catcode`@=11
\long\def\@makefntext#1{\noindent #1}
\newskip\tabcentering \tabcentering=1000pt plus 1000pt minus 1000pt
\def\MCH#1#2{\setbox0=\hbox{\raise#1\hbox{#2}}\smash{\box0}}% move char

\def\@evenfoot{}\def\@oddfoot{}

\def\@evenhead{\hbox to\textwidth{\footnotesize\rm\thepage \hfill
{\it Cheng Zhiyun}}} % authors name

\def\@oddhead{\hbox to \textwidth{\footnotesize{\it
A polynomial invariant of virtual knots} \hfill\thepage}}% abbreviate title

\catcode`@=12

\def\bc{\begin{center}}
\def\ec{\end{center}}
\def\no{\noindent}
\def\hang{\hangindent\parindent}
\def\textindent#1{\indent\llap{\qquad #1\ \ \enspace}\ignorespaces}
\def\ref{\par\hang\textindent}

\newtheorem{theorem}{Theorem}[section]

\newtheorem{lemma}[theorem]{Lemma}
\newtheorem{corollary}[theorem]{Corollary}
\newtheorem{example}[theorem]{Example}

\newtheorem{proposition}[theorem]{Proposition}

\begin{document}
\abovedisplayskip=6pt plus 1pt minus 1pt \belowdisplayskip=6pt
plus 1pt minus 1pt
%-------------------  First Head  -----------------------------------------
\thispagestyle{empty} \vspace*{-1.0truecm} \noindent
\vskip 10mm

\bc{\large\bf A polynomial invariant of virtual knots\\[2mm]
\footnotetext{\footnotesize The authors are supported by NSF 11171025}} \ec

\vskip 5mm
\bc{\bf Cheng Zhiyun\\
{\small School of Mathematical Sciences, Beijing Normal University
\\Laboratory of Mathematics and Complex Systems, Ministry of
Education, Beijing 100875, China
\\$($email: czy@mail.bnu.edu.cn$)$}}\ec

\vskip 1 mm

\noindent{\small {\small\bf Abstract} The aim of this paper is to introduce a polynomial invariant $f_K(t)$ for virtual knots. We show that $f_K(t)$ can be used to distinguish some virtual knot from its inverse and mirror image. The behavior of $f_K(t)$ under connected sum is also given. Finally we discuss which kind of polynomial can be realized as $f_K(t)$ for some virtual knot $K$.
\ \

\vspace{1mm}\baselineskip 12pt

\no{\small\bf Keywords} virtual knot; odd writhe polynomial\ \

\no{\small\bf MR(2000) Subject Classification} 57M25\ \ {\rm }}

\vskip 1 mm

%\noindent{\small {\small\bf Abstract} \ \

\vspace{1mm}\baselineskip 12pt

%\no{\small\bf Keywords} \ \

%\no{\small\bf MR(2000) Subject Classification}\ \ {\rm }}

\section{Introduction}
Virtual knot theory was proposed by Louis H. Kauffman in [2]. Classical knot theory studies the embeddings of circles in thickened $S^2$, as a generalization of classical knot theory, virtual knot theory studies the stabilized embeddings of circles in thickened surfaces of arbitrary genus. Another motivation of studying virtual knot theory comes from the representations of knot diagrams by oriented Gauss diagrams. Two virtual knots are equivalent if and only if their associated Gauss diagrams are equivalent modulo the corresponding Reidemeister moves. Since not all Gauss diagrams can be realized on the plane, virtual knots are represented by planar diagrams by introducing virtual crossings as well as real crossings.

An interesting question in virtual knot theory is how to detect whether a given virtual knot is classical $($has a diagram without virtual crossing$)$ or not? Many invariants have been introduced to give some obstructions for a virtual knot to be classical, see [4] for a good survey. One simple but useful invariant is the odd writhe $J(K)$, which was defined in [3]. Inspired by the warping polynomial introduced in [8], we define a new polynomial invariant $f_K(t)$ of $\mathbb{Z}[t, t^{-1}]$ for each virtual knot $K$, say the \emph{odd writhe polynomial}, since it will be shown that $f_K(\pm1)=J(K)$. Moreover we will show that $f_K(t)$ is really more powerful than $J(K)$, and some basic properties of $f_K(t)$ will be discussed. Finally a sufficient and necessary condition for a polynomial to be the odd writhe polynomial of some virtual knots is given.

\section{The definition of the odd writhe polynomial}
According to [5], a \emph{virtual link diagram} is a planar 4-valent graph endowed with some crossing information on each crossing point: either an overcrossing and undercrossing or a virtual crossing $($a 4-valent vertex with a small circle around it$)$. Two link virtual diagrams are \emph{equivalent} if there exists a sequence of generalized Reidemeister moves connecting them, see the figure below. Here $\Omega_1, \Omega_2, \Omega_3$ denote the classical Reidemeister moves, $\Omega'_1, \Omega'_2, \Omega'_3$ are the virtual versions of the classical Reidemeister moves and $\Omega^s_3$ represents the semivirtual version of the third Reidemeister move.
\begin{center}
\includegraphics{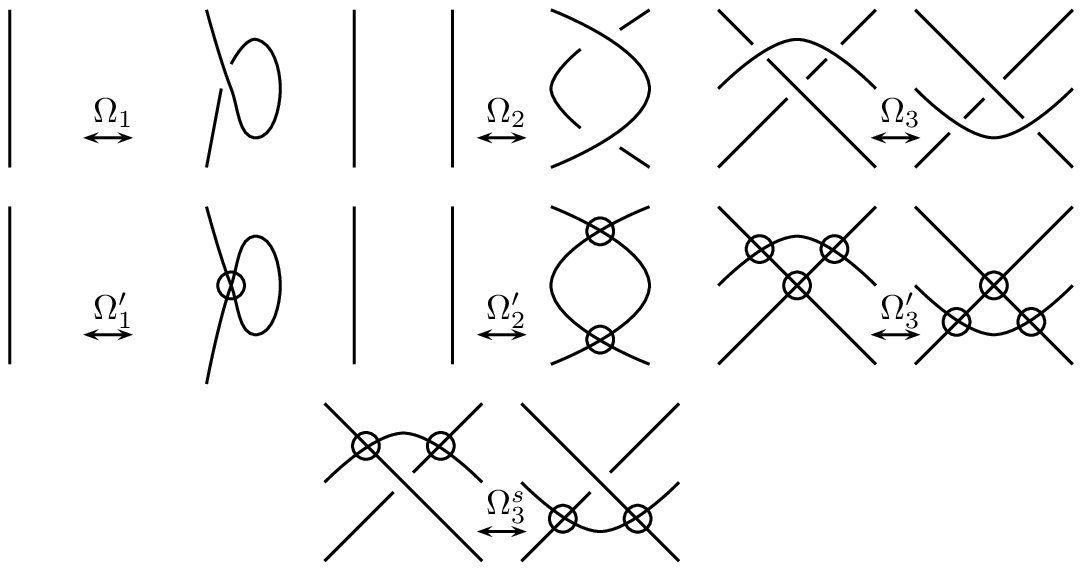} \centerline{\small Figure
1\quad}
\end{center}

From now on, $($virtual$)$ knot diagrams we mention below are all oriented. For each $($virtual$)$ knot diagram $K$, the \emph{Gauss diagram} of $K$ consists of a circle together with a chord connecting the preimages of each classical crossing point. To incorporate the information of overcrossing and undercrossing, the chords are endowed with an orientation from the preimage of the overcrossing to the preimage of the undercrossing. The sign of each chord is equal to the writhe of the corresponding crossing. Without loss of generalization, we assume that the circle in the Gauss diagram is anti-clockwise oriented. The figure below gives an example of virtual trefoil and its Gauss diagram.
\begin{center}
\includegraphics{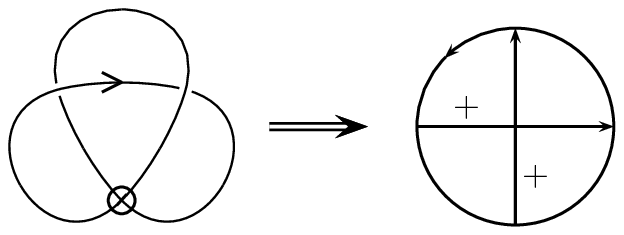} \centerline{\small Figure
2\quad}
\end{center}

Given a virtual knot diagram, there is a unique associated Gauss diagram. However given a Gauss diagram, the corresponding virtual diagrams are not unique. Note that $\Omega'_1, \Omega'_2, \Omega'_3$ and $\Omega^s_3$ all preserve the Gauss diagram, the following theorem was proved in [1].
\begin{theorem}$^{[1]}$
A Gauss diagram uniquely defines a virtual knot isotopy class.
\end{theorem}

Now we give a short review of the odd writhe of a virtual knot. For a classical knot diagram with $n$ crossings, there are totally $2n$ vertices on the circle of the associated Gauss diagram. It is easy to find that each chord flanks an even number of vertices in the Gauss diagram. Given a virtual knot diagram, we can obtain a Gauss diagram similarly $($without considering the virtual crossing points$)$. However in virtual knot theory, it is possible that there exists one chord in the Gauss diagram such that there are odd number of vertices on both sides of it. We name the corresponding real crossing an \emph{odd} crossing, the associated chord an \emph{odd} chord. Given a virtual knot diagram $K$, similar to [3] we use $Odd(K)$ to denote all the odd crossings of it. If a real crossing $c_j\overline{\in} Odd(K)$, we say it is \emph{even}. Let $w(c_i)$ be the writhe of the crossing point $c_i$, we usually abuse our notation, let $c_i$ also denote both the real crossing and the corresponding chord in the Gauss diagram. Then the \emph{odd writhe} of $K$ can be defined as
\bc
$J(K)=\sum\limits_{c_i\in Odd(K)}w(c_i)$.
\ec
It is not difficult to show that $J(K)$ is an invariant of virtual knots. As an example, the two crossing points in the virtual trefoil diagram are both odd, hence $J(K)=2$. Since classical knot has zero odd writhe, then we see that the virtual trefoil is non-classical and hence non-trivial.

Given a virtual knot diagram $K$ with real crossings points $\{c_1, \cdots, c_n\}$, they correspond to $2n$ vertices $\{c^+_1, c^-_1, \cdots, c^+_n, c^-_n\}$ on the circle of the Gauss diagram, where $c^+_i (c^-_i)$ denotes the preimage of the overcrossing $($undercrossing$)$ point of $c_i$. These $2n$ vertices will divides the circle into $2n$ arcs, say $\{a_1, \cdots, a_{2n}\}$. In order to define our polynomial invariant $f_K(t)$, we first assign an integer $N(a_i)$ to each arc $a_i$ as follows. Choose a point in $a_i$, go along the circle according to the orientation $($anti-clockwise$)$. For some chords $c_i$, we will meet $c^+_i$ earlier than $c^-_i$. Then we define $N(a_i)$ to be the sum of $w(c_i)$ for all the chords $\{c_i\}$ which satisfy this condition. It is easy to find that the assigned numbers of two adjacent arcs are different by 1, see the figure below:
\begin{center}
\includegraphics{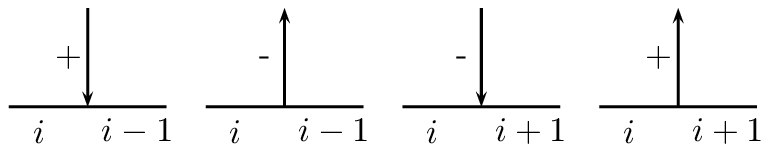} \centerline{\small Figure
3\quad}
\end{center}

Consider a chord $c_i$ in the Gauss diagram, the assigned numbers near $c^+_i$ and $c^-_i$ can be described as below:
\begin{center}
\includegraphics{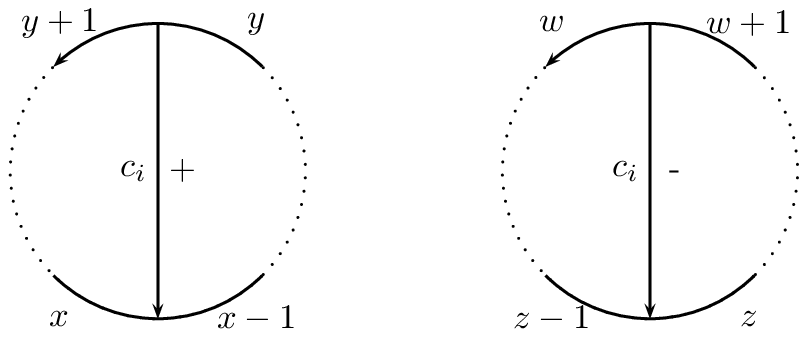} \centerline{\small Figure
4\quad}
\end{center}
Then we assign an integer $N(c_i)$ to each chord $c_i$ by
\bc
$N(c_i)=
\begin{cases}
x-y& \text{if $w(c_i)=+1$;}\\
z-w& \text{if $w(c_i)=-1$.}
\end{cases}$
\ec
Now we can define the \emph{odd writhe polynomial} of $K$ to be
\bc
$f_K(t)=\sum\limits_{c_i\in Odd(K)}w(c_i)t^{N(c_i)}.$
\ec

We note that when $c_i$ is odd, $N(c_i)$ is an even integer. It is obvious that $f_K(\pm1)=J(K)$. Similar to $J(K)$, $f_K(t)=0$ if $K$ contains no odd crossing, hence $f_K(t)$ can be used to detect non-classicality, and hence non-triviality. We will discuss more properties of $f_K(t)$ in Section 4. In the next section we will show that $f_K(t)$ is a virtual knot invariant. We end this section with an elementary but important property about $J(K)$.
\begin{lemma}
$J(K)$ is even for all virtual knots.
\end{lemma}
\begin{proof}
We prove this lemma by induction on the real crossing number of the virtual knot diagram. When the diagram $K$ contains 0 or 1 real crossing point, the result follows obviously. Assuming when $K$ contains $n$ real crossing points, the conclusion is correct. Now we consider the case $K$ contains $n+1$ real crossing points.

Let $\{c_1, \cdots, c_{n+1}\}$ be the chords in the Gauss diagram of $K$, consider the chord $c_{n+1}$. We continue our discussion in two cases:
\begin{itemize}
  \item $c_{n+1}$ is even. We use $K'$ to denote the virtual knot which is obtained by removing $c_{n+1}$ from the Gauss diagram of $K$ $($replace it by a virtual crossing on the diagram$)$. Then by induction, we have $J(K')$ is even. Since $c_{n+1}$ is not odd, there are even number of chords intersecting $c_{n+1}$ in the Gauss diagram of $K$. Each chord from these even number chords has different parities in $K'$ and $K$, while each chord from the rest has the same parity in $K'$ and $K$. Together with $c_{n+1}$ is not odd, hence $J(K)$ is different from $J(K')$ by an even integer, hence is also even.
  \item $c_{n+1}$ is odd. In this case, there are odd number of chords intersecting $c_{n+1}$ in the Gauss diagram of $K$, hence without considering the contribution of $w(c_{n+1})$, $J(K)$ and $J(K')$ have different parities. Together with $c_{n+1}$ is odd, it follows that $J(K)$ is an even integer.
\end{itemize}
\end{proof}

\section{The invariance of the odd writhe polynomial}
In this section we will prove the theorem below:
\begin{theorem}
The odd writhe polynomial $f_K(t)$ is a virtual knot invariant.
\end{theorem}
\begin{proof}
According to Theorem 2.1, it suffices to prove that $f_K(t)$ is invariant under the corresponding Reidemeister moves on the Gauss diagram. In [6], M. Polyak proved that all the classical Reidemeister moves can be realized by a generating set of four Reidemeister moves: $\{\Omega_{1a}, \Omega_{1b}, \Omega_{2a}, \Omega_{3a}\}$, see the figure below. Hence it suffices to show that $f_K(t)$ is invariant under $\Omega_{1a}, \Omega_{1b}, \Omega_{2a}$ and $\Omega_{3a}$.
\begin{center}
\includegraphics{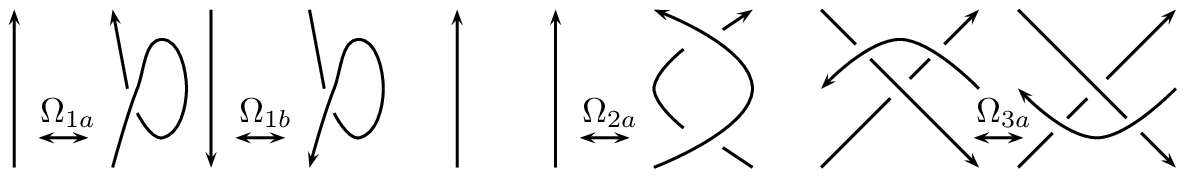} \centerline{\small Figure
5\quad}
\end{center}

First let us consider $\Omega_{1a}$ and $\Omega_{1b}$. It is easy to find that the chord representing the new crossing has no intersection with other chords. Let $a_i$ be the arc where the preimages of the new crossing located in. If $N(a_i)=k$, the figure below describes the corresponding transformation of $\Omega_{1a}$ and $\Omega_{1b}$ on the Gauss diagram.
\begin{center}
\includegraphics{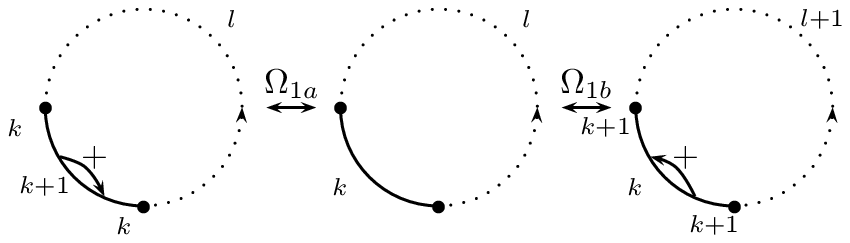} \centerline{\small Figure
6\quad}
\end{center}
For the $\Omega_{1a}$-move, the assigned number of all other arcs are preserved, since the new chord is not odd, hence it has no contribution to the odd writhe polynomial. The odd writhe polynomial is invariant. For the $\Omega_{1b}$-move, the assigned number of all other arcs are increased by one, according to the definition of $f_K(t)$, the odd writhe polynomial is also invariant.

Second we prove the odd writhe polynomial is preserved under $\Omega_{2a}$. Similarly let us consider the behavior of the $\Omega_{2a}$-move on the Gauss diagram. Note that if we use $c_m$ and $c_n$ to denote the two new crossing points $($and the associated chords$)$, then $c^+_m, c^+_n$ are both located in one arc $a_j$, and $c^-_m, c^-_n$ are also both located in one arc $a_i$. Without loss of generalization, we can suppose that $N(a_i)=i$ and $N(a_j)=j$. The figure below shows the corresponding transformation of $\Omega_{2a}$ on the Gauss diagram.
\begin{center}
\includegraphics{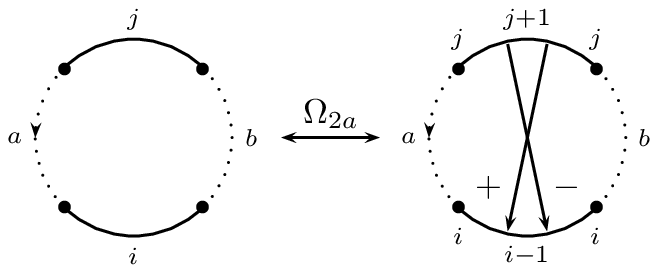} \centerline{\small Figure
7\quad}
\end{center}
On the right side of Figure 7, it is evident that either $c_m$ and $c_n$ are both odd, or $c_m$ and $c_n$ are both even. If they are both even, since the assigned number of other arcs are invariant, it follows that the odd writhe polynomial is invariant. If $c_m$ and $c_n$ are both odd, it suffices to consider their contribution to the odd writhe polynomial. According to the definition of $f_K(t)$, the contribution of $c_m$ and $c_n$ equals to $t^{i-j}-t^{i-j}=0$. As a result, the odd writhe polynomial is preserved under the $\Omega_{2a}$-move.

Finally let us consider the $\Omega_{3a}$-move. Similar as the above, we need to investigate the effect of $\Omega_{3a}$ on Gauss diagram. Different from the previous $\Omega_{1a}, \Omega_{1b}, \Omega_{2a}$, when considering the $\Omega_{3a}$-move, the Gauss diagram has two different possibilities, according to the structure of the diagram outside the local given diagram in Figure 5. The figure below gives a description of the behavior of $\Omega_{3a}$ on these two cases.
\begin{center}
\includegraphics{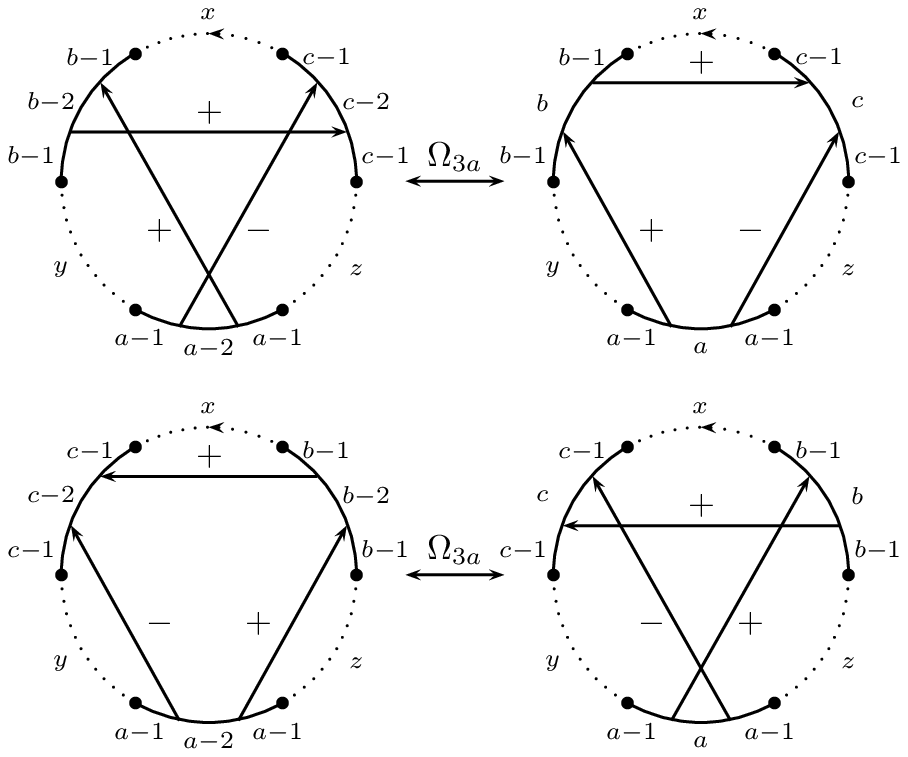} \centerline{\small Figure
8\quad}
\end{center}
From the figure above it is not difficult to observe that for each chord $c_i$, the assigned number $N(c_i)$ is invariant. Because the writhe and the parity of each crossing are also preserved, the odd writhe polynomial is kept under the $\Omega_{3a}$-move. In conclusion, the odd writhe polynomial is invariant under $\Omega_{1a}, \Omega_{1b}, \Omega_{2a}$ and $\Omega_{3a}$, hence the proof is finished.
\end{proof}

We end this section with a simple example. Consider the virtual trefoil given in Figure 2. After assigning all the arcs and chords, we have the figure below, it follows that $f_K(t)=t^2+1$.
\begin{center}
\includegraphics{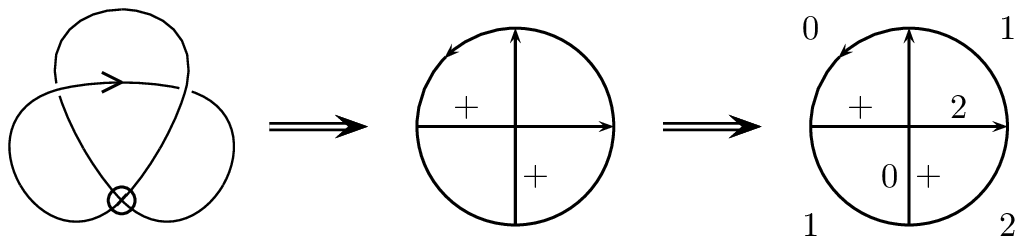} \centerline{\small Figure
9\quad}
\end{center}

\section{Some properties of the odd writhe polynomial}
\subsection{Compare with the odd writhe}
Although the odd writhe polynomial is a generalization of the odd writhe, if $J(K)=0$ always implies $f_K(t)=0$, there is little interest of $f_K(t)$. Fortunately this is not the truth, the following instance gives an example of this.

\begin{example}
\emph{The odd writhe of the virtual knot below is trivial, but its odd writhe polynomial is non-trivial.}
\end{example}
\begin{center}
\includegraphics{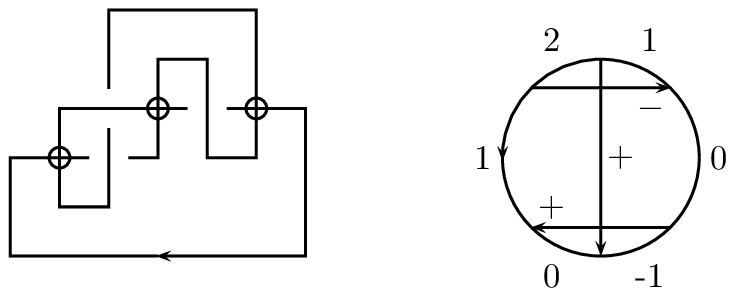} \centerline{\small Figure
10\quad}
\end{center}
It is obvious that the Gauss diagram above contains two odd chords, one has writhe +1, the other one has writhe -1, hence its odd writhe is 0. On the other hand, the two odd chords will contribute $+t^{1-(-1)}$ and $-t^{1-1}$ to the odd writhe polynomial respectively, therefore the odd writhe polynomial of it equals to $t^2-1$, which is non-trivial.

Given a virtual knot, suppose its odd writhe polynomial is $f_K(t)=\sum\limits a_it^i$, we define the degree of $f_K(t)$ to be
\bc
Deg $f_K(t)=\max\{|i| \mid a_i\neq0\}$.
\ec
In virtual knot theory, given a virtual knot diagram $D$ there are three kinds of crossing number: the number of real crossings $c_r(D)$, the number of virtual crossings $c_v(D)$ and the number of total crossings $c(D)$. Among all the virtual diagrams of a virtual knot $K$, the minimal values of them are invariants of $K$, say the real crossing number $c_r(K)$, the virtual crossing number $c_v(K)$ and the crossing number $c(K)$ respectively. The relations of these three invariants and the lower bounders of them are very interesting and important problems in virtual knot theory. The next proposition gives a simple lower bounder of the real crossing number.
\begin{proposition}
Given a virtual knot $K$, we have $c_r(K)\geqslant$ Deg $f_K(t)$.
\end{proposition}

\begin{proof}
The proof is evident, just consider the chord which realizes the max $|i|$, according to Figure 3, we have $c_r(D)\geqslant$ Deg $f_K(t)$ for any virtual diagram $D$ of $K$. The conclusion follows.
\end{proof}

It is easy to find that the absolute value of the odd writhe is also a lower bounder of $c_r$, which means that $c_r(K)\geqslant|J(K)|$. For example, the odd writhe of the virtual trefoil is 2, hence its real crossing number is exactly 2. In general, if the absolute value of the odd writhe of a diagram equals to the real crossing number of this diagram, then we can conclude that $c_r(D)=c_r(K)=|J(K)|$. Similarly if the degree of the odd writhe polynomial of a knot diagram equals to the real crossing number of this diagram, we also obtain that $c_r(D)=c_r(K)=$ Deg $f_K(t)$. However the condition for $c_r(K)=|J(K)|$ is very sharp, in fact $c_r(K)=|J(K)|$ if and only if there exists a diagram of $K$ with real crossing number $c_r$ such that all real crossing points are odd and the writhes of them are all the same. The next example shows that from the viewpoint of the odd writhe polynomial, the condition is much weaker.

\begin{example}
\emph{The odd writhe of the virtual knot below is trivial, but we can deduce that the real crossing number of it is 4, from the odd writhe polynomial.}
\end{example}
\begin{center}
\includegraphics{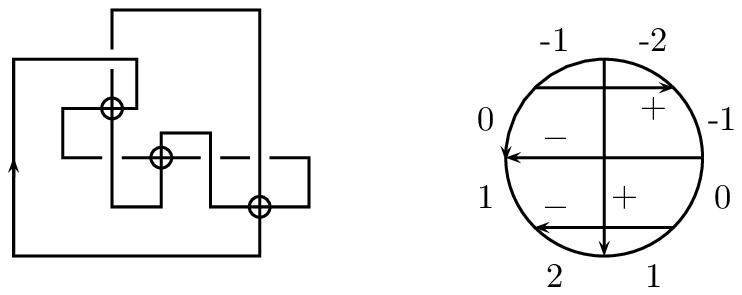} \centerline{\small Figure
11\quad}
\end{center}
Now all the four chords of the Gauss diagram above are odd, hence its odd writhe is 0. However the odd writhe polynomial of it equals to $t^{2-(-2)}-t^{2-0}-t^{1-(-1)}+t^{-1-(-1)}=t^4-2t^2+1$. It follows that the degree of $f_K(t)$ is 4, hence the real crossing number of the knot above is exactly 4.

\textbf{Remark} In fact a virtual knot $K$ satisfies $c_r(K)=$ Deg $f_K(t)=2k$ $(k>0$. For the case of $k<0$, the corresponding diagram can be obtained in the similar way$)$ only if $K$ has a diagram as follows:
\begin{center}
\includegraphics{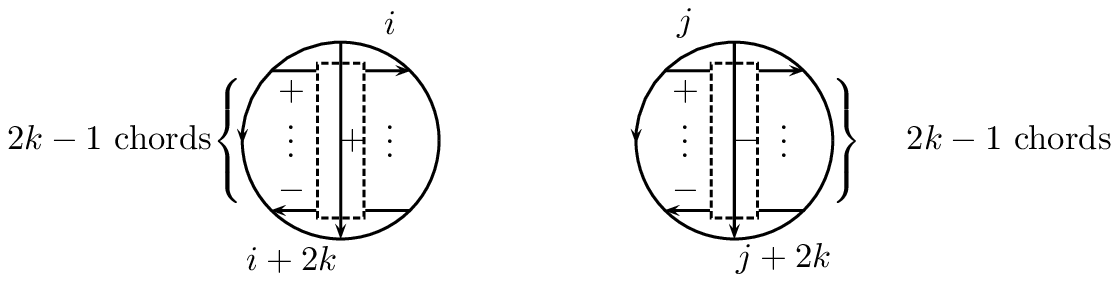} \centerline{\small Figure
12\quad}
\end{center}
On the left side of Figure 12, the positive vertical chord contributes the monomial $t^{2k}$ to the odd writhe polynomial. All other chords intersect the vertical one transversely, and all the chords that are oriented from left to right have writhe $+1$, the rests have writhe $-1$. The dotted frame means the positions of the $2k-1$ chords maybe very complicated, not need to be parallel. For the right side figure, the vertical chord has writhe $-1$, hence contributes $-t^{2k}$ to the odd writhe polynomial, all other chords are the same to the left figure. Hence we do not need that all the crossing points are odd or all of them have the same writhe.

\subsection{The relation between $f_{K^*}(t)$ and $f_K(t)$}
Given an oriented virtual knot $K$, by the \emph{inverse} of $K$, we mean the knot obtained from $K$ by reversing its orientation. Here we use $K^*$ to denote it. Now we want to investigate the relation between $f_{K^*}(t)$ and $f_K(t)$. In fact we have the proposition below.
\begin{proposition}
$f_{K^*}(t)=f_K(t^{-1})\cdot t^2$.
\end{proposition}

\begin{proof}
Choose a virtual diagram of $K$, say $D$. Consider the Gauss diagrams of $D$ and $D^*$, note that the writhe and crossing information of each real crossing point are preserved, the figure below describes the local transformation on the Gauss diagram, where $w$ denotes the writhe of the diagram $D$.
\begin{center}
\includegraphics{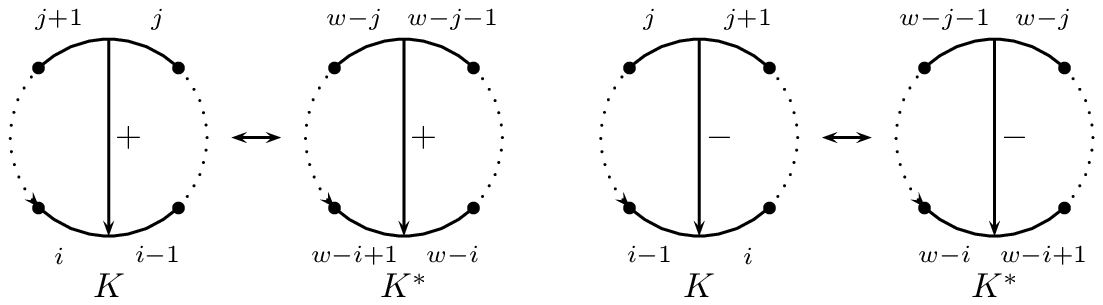} \centerline{\small Figure
13\quad}
\end{center}
From the figure above, it easy to find that a crossing is odd in $D$ if and only if it is also odd in $D^*$, and their contribute to the odd writhe polynomial are $\pm t^{i-j}$ and $\pm t^{j-i+2}$ respectively. It follows that $f_{K^*}(t)=f_K(t^{-1})\cdot t^2$.
\end{proof}

As an immediate corollary, we have:
\begin{corollary}
Given a virtual knot $K$, if $f_K(t)\neq a_nt^n+a_{n-2}t^{n-2}+\cdots+a_4t^4+a_2t^2+a_2+a_4t^{-2}+\cdots+a_{n-2}t^{4-n}+a_nt^{2-n}$, then $K\neq K^*$.
\end{corollary}

As an example, consider the virtual knot in Example 4.3, its odd writhe polynomial is $t^4-2t^2+1$, and the odd writhe polynomial of its inverse is $t^2-2+t^{-2}$. Hence the knot in Example 4.3 is non-invertible. Note that the odd writhe of a virtual knot is kept under reversing the orientation.

\subsection{The relation between $f_{\overline{K}}(t)$ and $f_K(t)$}
Given a virtual knot $K$, we define the \emph{mirror image} of $K$, say $\overline{K}$, to be the knot which is obtained from $K$ by switching all the real crossing points. The following proposition shows that the odd writhe polynomial sometimes can be used to distinguish a virtual knot from its mirror image.
\begin{proposition}
$f_{\overline{K}}(t)=-f_K(t^{-1})\cdot t^2$.
\end{proposition}

\begin{proof}
Consider the Gauss diagrams of $K$ and $\overline{K}$. Since the over-crossing and under-crossing will be switched and the writhe of each crossing will be changed, the local transformation on the Gauss diagram is given below. As above, here $w$ also denotes the writhe of $K$.
\begin{center}
\includegraphics{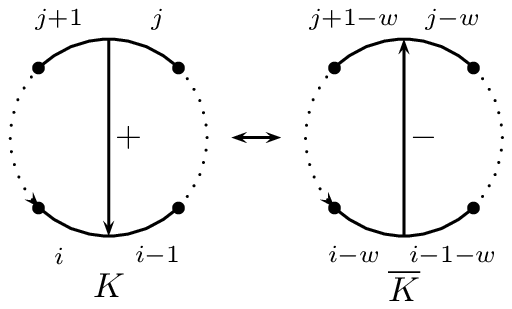} \centerline{\small Figure
14\quad}
\end{center}
From the figure above, if one chord contributes $\pm t^{i-j}$ to the odd writhe polynomial of $K$, the corresponding chord will contribute $\mp t^{j-i+2}$ to the odd writhe polynomial of $\overline{K}$. Note that a chord is odd in $K$ if and only if the corresponding chord is also odd in $\overline{K}$, the conclusion follows.
\end{proof}

The following direct corollary offers an obstruction for a virtual knot to be amphicheiral.
\begin{corollary}
Given a virtual knot $K$, if $f_K(t)\neq a_nt^n+a_{n-2}t^{n-2}+\cdots+a_4t^4+a_2t^2-a_2-a_4t^{-2}-\cdots-a_{n-2}t^{4-n}-a_nt^{2-n}$, then $K\neq \overline{K}$.
\end{corollary}

Together with Proposition 4.4 and Proposition 4.6, we have:
\begin{corollary}
$f_K(t)=-f_{\overline{K}^*}(t)$.
\end{corollary}

Now let us consider the virtual knot in Example 4.3 again, it is easy to find that the odd writhe polynomial of its mirror image is $-t^2+2-t^{-2}$. Hence $K$ is not equivalent to its mirror image. Note that although the odd writhe sometimes can distinguish a virtual knot from its mirror $($since $J(K)=-J(\overline{K})$$)$, in this case the odd writhe is trivial.

\subsection{The behavior under connected sum}
In virtual knot theory, the connected sum is not well-defined $($even for oriented virtual knots$)$, the result of a connected sum strongly depends on the choice of the place where the connection is made. For this reason, the notation $K_1\#K_2$ does not make sense in general. Very interesting, although the connected sum of virtual knots is not well-defined, the odd writhe polynomial is well defined under the ``connected sum" operation. Equivalently speaking, the odd writhe polynomial does not depend on the choice of the place where the connecting is made. In fact, we have the following proposition.
\begin{proposition}
Given two virtual knots $K_a$ and $K_b$, we have $f_{K_a\#K_b}(t)=f_{K_a}(t)+f_{K_b}(t)$. Here $K_a\#K_b$ denotes one connected sum of $K_a, K_b$ with an arbitrarily chosen of the connection place.
\end{proposition}

\begin{proof}
Choose two diagrams of $K_a$ and $K_b$, say $D_a$ and $D_b$. Assume the connecting is made at two arcs of $D_a$ and $D_b$, say $a_1$ and $a_2$ with $N(a_1)=i$ and $N(a_2)=j$. Then the Gauss diagram of $D_a\#D_b$ is described as below.
\begin{center}
\includegraphics{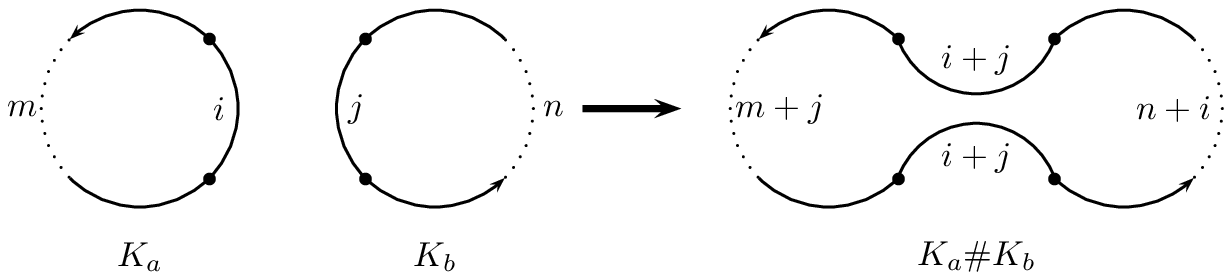} \centerline{\small Figure
15\quad}
\end{center}
Note that the assigned numbers of the arcs in $D_a$ and $D_b$ are increased by $j$ and $i$ respectively, hence the assigned number of each chord is preserved. It follows that $f_{K_a\#K_b}(t)=f_{K_a}(t)+f_{K_b}(t)$. The proof is finished.
\end{proof}

\section{The characterization of the odd writhe polynomial}
As an application of the properties given in Section 4, now let us consider the characterization problem of the odd writhe polynomial. We want to identify which kind of polynomial of $\mathbb{Z}[t, t^{-1}]$ can be realized as the odd writhe polynomial of a virtual knot.

First we recall that according to the definition of the odd writhe polynomial, we only account all the odd chords, hence the exponents in the polynomial are all even. It follows that the polynomial can be written as
\bc
$f=a_{2n}t^{2n}+a_{2n-2}t^{2n-2}+\cdots+a_{2n-2m}t^{2n-2m}$ $(m\geq0)$.
\ec
Second by Lemma 2.2, the odd writhe invariant is always even, since $J(K)=f_K(\pm1)$, we conclude that the coefficients of the odd writhe polynomial should satisfy
\bc
$\sum\limits_{i=2n-2m}^{2n}a_i=0$ $($mod 2$)$.
\ec

The following theorem tells us the two conditions above are not only necessary but also sufficient.
\begin{theorem}
A polynomial $f$ can be realized as the odd writhe polynomial of a virtual knot if and only if $f$ can be written as
\bc
$f=a_{2n}t^{2n}+a_{2n-2}t^{2n-2}+\cdots+a_{2n-2m}t^{2n-2m}$ $(m\geq0)$,
\ec
where the coefficients satisfy
\bc
$\sum\limits_{i=2n-2m}^{2n}a_i=0$ $($\emph{mod 2}$)$.
\ec
\end{theorem}

\begin{proof}
The necessary part has been mentioned, it suffices to prove the sufficient part. We want to construct a virtual knot $K$ such that
\bc
$f_K(t)=a_{2n}t^{2n}+a_{2n-2}t^{2n-2}+\cdots+a_{2n-2m}t^{2n-2m}$ $(m\geq0)$.
\ec

First of all, let us consider these virtual knots $L_{2k}$ $(k\geq1)$ as follows. Unlike the Gauss diagrams given in Figure 12, at present the $2k-1$ horizontal chords are parallel.
\begin{center}
\includegraphics{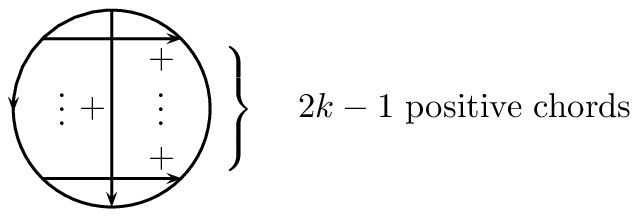} \centerline{\small Figure
16\quad}
\end{center}
Note that the Gauss diagram above contains $2k$ positive chords and all of them are odd chords. Direct calculation shows that $f_{L_{2k}}(t)=t^{2k}+2k-1$. Together with Proposition 4.4, Proposition 4.6 and Corollary 4.8, we have
\begin{align*}
f_{L_{2k}}(t)&=t^{2k}+2k-1, \\
f_{L^*_{2k}}(t)&=t^{-2k+2}+(2k-1)t^2, \\
f_{\overline{L_{2k}}}(t)&=-t^{-2k+2}-(2k-1)t^2, \\
f_{\overline{L_{2k}}^*}(t)&=-t^{2k}-2k+1. \\
\end{align*}
Consider the virtual knot $L=a_{2n}L_{2n}\#\cdots\#a_4L_4\#a_{-2}L^*_4\#\cdots\#a_{2n-2m}L^*_{2m-2n+2}$, if a coefficient before $L_i$ is negative, then we replace $L_i$ by $\overline{L_i}^*$. Hence the representation of $L$ makes sense. According to Proposition 4.9, we have
\bc
$f_L(t)=a_{2n}t^{2n}+\cdots+a_4t^4+a_{-2}t^{-2}+\cdots+a_{2n-2m}t^{2n-2m}+b_2t^2+b_0$,
\ec
where
\begin{align*}
&b_2=\pm3a_{-2}+\cdots+\pm(2m-2n+1)a_{2n-2m}, \\
&b_0=\pm3a_4+\cdots+\pm(2n-1)a_{2n}. \\
\end{align*}
Hence it suffices to construct a knot with odd writhe polynomial $(a_2-b_2)t^2+a_0-b_0$, note that $a_2-b_2+a_0-b_0=\sum\limits_{i=2n-2m}^{2n}a_i=0$ $($mod 2$)$.

Now let us consider the two virtual knots in Figure 9 and Figure 10, we use $M$ and $N$ to denote them here. It has been shown that
\bc
$f_M(t)=t^2+1$, $f_N(t)=t^2-1$.
\ec
As a result, we have
\bc
$f_{(\frac{a_2-b_2+a_0-b_0}{2})M\#(\frac{a_2-b_2-a_0+b_0}{2})N}(t)=(a_2-b_2)t^2+a_0-b_0.$
\ec
As before, if $a_2-b_2+a_0-b_0$ $(a_2-b_2-a_0+b_0)$ is negative, then we replace $M$ $(N)$ by $\overline{M}^*$ $(\overline{N}^*)$. Recall that $a_2-b_2+a_0-b_0$ and $a_2-b_2-a_0+b_0$ are both even, hence $(\frac{a_2-b_2+a_0-b_0}{2})M\#(\frac{a_2-b_2-a_0+b_0}{2})N$ makes sense, here the connected sum means to choose arbitrary connecting place.

In conclusion, the virtual knot $K=L\#(\frac{a_2-b_2+a_0-b_0}{2})M\#(\frac{a_2-b_2-a_0+b_0}{2})N$ satisfies our requirement. The proof is finished.
\end{proof}

Finally we want to spend a little time in discussing the case of classical knots. We have mentioned that the odd writhe polynomial is trivial on classical knots. It is natural to ask whether we can define a similar polynomial invariant for classical knots. For example, we wish it can tell the difference between a knot and its inverse. Given a knot diagram $K$, since there is no odd crossing point at present, we have to consider the polynomial
\bc
$f_K(t)=\sum\limits_{c_i\in C(K)}w(c_i)t^{N(c_i)},$
\ec
here $C(K)$ denotes the crossing points of a diagram $K$, and $N(c_i)$ is defined as before. However $f_K(t)$ is not an invariant. In fact it is easy to find that $f_K(t)$ is invariant under the second and the third Reidemeister move, but the first Reidemeister move can change the value of it. In order to define an invariant, we can consider the polynomial
\bc
$F_K(t)=f_K(t)-w(K)t,$
\ec
here $w(K)$ denotes the writhe of the diagram. It is easy to check that $F_K(t)$ defined above is invariant under the first Reidemeister move. Since the writhe is kept under the second and the third Reidemeister move, it follows that $F_K(t)$ is a polynomial invariant of classical knots. Unfortunately, with a little computation we find that $F_K(t)$ is a trivial invariant, i.e. $F_K(t)=0$ for all classical knots. In other words, $f_K(t)=w(K)t$ for all classical knot diagrams. In fact given a classical knot diagram, it is not difficult to observer that $N(c_i)\equiv1$ for each crossing point $c_i$, hence the result follows.

\end{document}